\documentclass[oneside,english]{amsart}
\usepackage[T1]{fontenc}
\usepackage[latin9]{inputenc}
\usepackage{mathrsfs}
\usepackage{amsthm}
\usepackage[all]{xy}

\makeatletter
\numberwithin{equation}{section}
\numberwithin{figure}{section}
\theoremstyle{plain}
\newtheorem{thm}{\protect\theoremname}[section]
  \theoremstyle{definition}
  \newtheorem{defn}[thm]{\protect\definitionname}
  \theoremstyle{remark}
  \newtheorem{rem}[thm]{\protect\remarkname}
  \theoremstyle{definition}
  \newtheorem{example}[thm]{\protect\examplename}
  \theoremstyle{plain}
  \newtheorem{prop}[thm]{\protect\propositionname}
  \theoremstyle{plain}
  \newtheorem{cor}[thm]{\protect\corollaryname}
  \theoremstyle{plain}
  \newtheorem{lem}[thm]{\protect\lemmaname}

\makeatother

\usepackage{babel}
  \providecommand{\corollaryname}{Corollary}
  \providecommand{\definitionname}{Definition}
  \providecommand{\examplename}{Example}
  \providecommand{\lemmaname}{Lemma}
  \providecommand{\propositionname}{Proposition}
  \providecommand{\remarkname}{Remark}
\providecommand{\theoremname}{Theorem}

\begin{document}

\title{Broad posets, trees, and the dendroidal category}

\author{Ittay Weiss}
\begin{abstract}
An extension of order theory is presented that serves as a formalism
for the study of dendroidal sets analogously to way the formalism
of order theory is used in the study of simplicial sets. 
\end{abstract}
\maketitle

\section{Introduction}

In algebraic topology a space $X$ is often replaced by its singular
complex $S(X)$ which is defined as follows. For every $n\ge0$ let
$\Delta^{n}$ be the standard $n$-dimensional simplex. The singular
complex $S(X)$ has, at each dimension $n\ge0$, the set $S(X)_{n}=\{f:\Delta^{n}\to X\mid f\mbox{ is continuous}\}$.
A famous result due to Quillen shows that for homotopy theoretic purposes
it makes little difference whether one works directly with the space
$X$ or its singular complex $S(X)$. The advantage of working with
$S(X)$ is that it is a completely combinatorial object belonging
to the theory of simplicial sets. The combinatorics within a simplicial
set is governed by the polytopal interrelations between the standard
$n$-simplices $\{\Delta^{n}\mid n\ge0\}$. As is well known, these
interrelations are equivalent to those between all finite, non-empty,
linearly ordered sets. The latter observation brings to the theory
of simplicial sets, and thus to topology, the very rigorous and algebraic
formalism of order theory. 

Recently (\cite{MR2366165}), the concept of simplicial set was generalized
to that of dendroidal set. The context of the generalization is operadic
rather than topological but the general aim is the same: to provide
for combinatorial models of operads and $\infty$-operads. The combinatorics
within a dendroidal set is governed by the interrelations between
finite trees. The theory of dendroidal sets is then seen to extend
that of simplicial sets by viewing every finite linear order as a
linear tree. 

The aim of this work is to present an extension of order theory that
stands in the same relation to dendroidal sets as order theory does
to simplicial sets. We exemplify how the extension we present can
be used to argue about trees, and thus dendroidal sets, in a way that
is analogous to the use of posets in arguments about simplicial sets.
We thus provide a rigorous and algebraic formalism playing the same
role in the theory of dendroidal sets as order theory does in the
theory of simplicial sets. 

Though we develop just that part of the theory of the extension of
order theory that is needed for its applicability to dendroidal sets,
we note that the theory appears to be interesting in its own right.

\subsection*{Plan of the paper}

Section \ref{sec:Trees-and-operads} contains an intuitive description
of trees and operads stressing the aspects that are relevant to the
formalism presented. Section \ref{sec:Broad-posets} then contains
the extension of order theory which is the ambient category where
the main objects of study, dendroidally ordered sets, presented in
Section \ref{sec:Dendroidally-ordered-sets}, reside. Section \ref{sec:The-correspondence-between}
then classifies dendroidally ordered sets in terms of trees, with
Section \ref{sec:Fundamental-structure-of} closing the article with
a proof, completely within the formalism developed below, of a fundamental
decomposition result in the dendroidal category.

\section{\label{sec:Trees-and-operads}Trees and operads}

Trees and operads are presented in a rather intuitive fashion meant
to immediately make available the key ideas relevant to the following
sections. For more detailed accounts of operads given in the spirit
of dendroidal sets the reader is referred to \cite{MR2797154} or
\cite{Weiss:2010fk}.

\subsection{Trees}

By a tree we mean a rooted, either planar or non-planar, tree with
leaves and stumps. Various definitions of tree exist in the literature
and perhaps the two most common approaches are to define a tree as
a special graph or as a special topological space. However, even within
a single framework there are many possibilities for a formalization
of the tree concept. For instance, it is common to define a tree as
consisting of a set $E$ of edges and a set $V$ of vertices together
with certain incidence relations and some conditions. But, it is also
possible (see \cite{Weiss:2010fk}) to dispose of the set $V$ and
capture vertices as a by-product of a certain structure only on the
set of edges. Expectedly, different formalisms have virtues in different
situations. 

The picture 
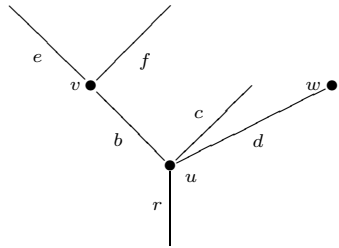
\[
\xymatrix{*{\,}\ar@{-}[dr]_{e} &  & *{\,}\ar@{-}[dl]^{f}\\
\,\ar@{}[r]|{\,\,\,\,\,\,\,\,\,\,\,\,\,\, v} & *{\bullet}\ar@{-}[dr]_{b} &  & *{\,}\ar@{-}[dl]_{c}\ar@{}[r]|{\,\,\,\,\,\,\,\,\,\,\,\, w} & *{\bullet}\ar@{-}[dll]^{d}\\
 &  & *{\bullet}\ar@{-}[d]_{r} & \,\ar@{}[l]^{u\,\,\,\,\,\,\,\,\,\,\,}\\
 &  & *{\,}
}
\]
of a tree exemplifies all of the features of interest to us. It consists
of six \emph{edges} of which the one labelled $r$ is the \emph{root}.
It has three \emph{vertices} marked by $\bullet$ and each vertex
has a set of \emph{incoming} edges and single \emph{outgoing} edge
leading to the root. One of these vertices, labelled $w$, contains
no incoming edges and is called a \emph{stump}. The edges labelled
$c,e,f$ have no vertex at their top and are called \emph{leaves}.
The picture can be taken as that of either a \emph{planar} tree or
a \emph{non-planar} one, with the crucial difference being whether
or not the order of the incoming edges at each vertex is important
or not, namely, in the planar case it is important and in the non-planar
one it is not. Certain parts of the tree lie at its outermost layer,
such as the pair of edges $\{e,f\}$ as well as the stump $w$. Such
regions of the tree are called \emph{external clusters}. Intuitively,
these are parts of the tree that can be trimmed by removing a single
vertex completely. Edges that belong to the inner layers of the tree
are called \emph{inner edges}. More precisely, these are edges with
a vertex at each end, such as edges $b$ and $d$. Intuitively, an
inner edge is an edge that can be contracted to merge together two
vertices. Other aspects visible in the tree above are that every edge
that is not a leaf has \emph{children} and that every edge other than
the root is a child of a unique \emph{parent} edge. Every two edges
admit a \emph{join}, an edge which is the first \emph{common ancestor}
of the given edges. For instance, the join of $e$ and $f$ is $b$
while that of $e$ and $c$ is $r$. Lastly, an intuitive feature
of trees is that two trees can be \emph{grafted} by placing one on
top of the other and identifying the root in one with a leaf in the
other.

\subsection{Operads}

By an operad we mean either a symmetric or planar coloured operad,
also known as a multicategory. Intuitively, it consists of a class
of labelled trees (planar ones for planar operads and non-planar ones
for symmetric operads) such that the edges in a tree are labelled
by objects and the vertices are labelled by multivariable arrows.
The various labelled trees must satisfy a consistency condition that
basically says that each multivariable arrow has a unique input and
output. Moreover, the class of labelled trees must be saturated, meaning
that every finite combinations of multivariable arrows with matching
inputs and outputs occur as a labelled tree. On top of that structure
there is then a composition operation that turns one such labelled
tree into a labelled tree with just one vertex and having the same
number of leaves as the original tree. For the composition there is
an associativity condition that says that starting with a single labelled
tree, composing it in one go or composing any subtrees of it first
will result in the same composition (and there are also identity constraints
which we neglect in this intuitive explanation). As such, any tree
naturally gives rise to an operad by generating a free one. The objects
are then the edges of the tree and the arrows are freely generated
by the vertices in the tree. Stumps are then interpreted as constants. 

We mention a few trees that play an important role in the theory.
A tree $L_{n}$ of the form
\[
\xymatrix{*{}\ar@{-}[d]\\
*{\bullet}\ar@{..}[d]\\
*{\bullet}\ar@{-}[d]\\
*{\bullet}\ar@{-}[d]\\
*{}
}
\]
with one leaf and only unary vertices is a \emph{linear tree}. The
linear trees stand to all trees in the same way that ordinary functions
stand to multivariable functions. The special case of the tree $L_{0}$

\[
\xymatrix{*{}\ar@{-}[d]\\
*{}
}
\]
consisting of just one edge and no vertices is denoted by $\eta$
and is the only tree, up to isomorphism, whose root is also a leaf. 

A tree $C_{n}$ of the form 
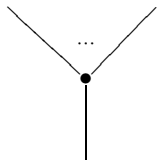
\[
\xymatrix{*{}\ar@{-}[dr] & *{}\ar@{}[d]|{\cdots} & *{}\ar@{-}[dl]\\
 & *{\bullet}\ar@{-}[d]\\
 & *{}
}
\]
that has just one vertex and $n$ leaves is called an \emph{$n$-coroll}a.
The corollas can be seen as the building blocks of all trees as any
tree is the grafting of corollas. A related remark is that any tree
$T$ admits an essentially unique decomposition as the grafting $C_{n}\circ(T_{1},\cdots,T_{n})$
where each $T_{i}$ is the subtree of $T$ having as its root the
$i$-th incoming edge to the root of $T$. 

In the study of operads, and particularly $\infty$-operads, the dendroidal
category plays a prominent role (see \cite{MR2366165,MR2508925,MR2797154,Weiss:2010fk,MR2805991,Cisinski:2010fk,Cisinski:2011uq}).
It comes in two flavours depending on whether one studies planar or
symmetric operads. The planar dendroidal category $\Omega_{\pi}$
has as objects all planar trees and as arrows all maps of operads
between the operads generated by the trees. Similarly, the non-planar
dendroidal category $\Omega$ has non-planar trees as objects and
symmetric operad maps as arrows. The interrelations between the trees
in (each variant of) the dendroidal category is what we call the operadic
tree combinatorics. 

To illustrate what we achieve below, consider the following. A common
definition of tree that uses the language of order theory is as a
poset $P$ that satisfies that for every $x\in P$ the down set $x_{\downarrow}=\{y\in P\mid y\le x\}$
is well-ordered. This definition does not capture the operadic tree
combinatorics. For instance, from the operadic point of view, an $n$-corolla
$C_{n}$ has precisely $n+1$ subtrees, all having just a single edge
(which represent the inputs and output of a multivariable arrow).
However, using the definition just mentioned, $C_{n}$ would also
have $n$ linear subtrees with two edges (which do not correspond
to anything one can obtain from a multivariable arrow).

Another formalism of trees which exhibits the same kind of difficulty
is that given in \cite{MR1301191} where a tree is defined as a topological
space. Recently, a formalism of trees in terms of polynomial functors
was given in \cite{MR2764874} which does capture the operadic tree
combinatorics. 

The aim of this work, which expands on ideas introduced by the author
briefly in \cite{Weiss:2010fk}, is to develop an order theoretic
formalism in which all of the intuitive properties of trees above
follow from just three axioms and such that the order preserving functions
capture the combinatorics of trees relevant to operads. Such a formalism
allows for very precise arguments about dendroidal sets that do not
rely on sometimes vague intuition about trees and can be useful in
other places where tree formalisms are needed.

\section{\label{sec:Broad-posets}Broad posets}

This section presents the notion of broad poset, exhibits ordinary
posets as a slice of broad posets, relates the latter to operads,
and establishes a closed symmetric monoidal structure by means of
a suitable tensor product of broad posets.

\subsection{Definition of broad posets and their relation to operads}

For a set $A$ we denote by $A^{\cdot}$ the free monoid on $A$ with
unit $\epsilon$. The free commutative monoid $A^{+}$ is obtained
from $A^{\cdot}$ by an obvious abelianization process. We use the
same notation $a\cdot b$ to indicate both the operation in $A^{\cdot}$
as well as in $A^{+}$. We identify $A$ in either $A^{\cdot}$ or
$A^{+}$ in the obvious way. 
\begin{defn}
A \emph{commutative broad relation} is a pair $(A,R)$ where $A$
is a set and $R$ is a subset of $A^{+}\times A$. A \emph{non-commutative
broad relation }is a pair $(A,R)$ where $A$ is a set and $R$ is
a subset of $A^{\cdot}\times A$. 
\end{defn}
All of the definitions and results below come in a commutative as
well as a non-commutative flavour, with the formulation essentially
unchanged. Thus, we use the notation $A^{*}$ to mean that it can
be replaced, throughout an entire definition or result, by either
$A^{\cdot}$ or $A^{+}$. The same convention holds for the use of
the term 'broad relation'. It can be replaced throughout a definition
or result by either 'non-commutative broad relation' or 'commutative
broad relation'. The following is an instance of this convention.
As is common with ordinary relations, for $a_{1}\in A^{*}$ and $a_{2}\in A$,
we write $a_{1}Ra_{2}$ to mean $(a_{1},a_{2})\in R$, for any broad
relation $R$. We also write $a_{2}\in a_{1}$ to indicate that $a_{2}$
occurs in $a_{1}$ (as a factor in the non-commutative case and as
a summand in the commutative case). 
\begin{defn}
A \emph{broad poset} is a broad relation $(A,R)$ such that, for all
$n\ge0$, $a_{1},\cdots,a_{n},a,a'\in A$ and $b_{1},\cdots,b_{n}\in A^{*}$,
the following conditions hold.\end{defn}
\begin{itemize}
\item Reflexivity: $aRa$. 
\item Transitivity: If $a_{1}\cdot\cdots\cdot a_{n}Ra$ and, for all $1\le i\le n$,
$b_{i}Ra_{i}$ hold then $b_{1}\cdot\cdots\cdot b_{n}Ra$ holds.
\item Anti-symmetry: If $aRa'$ and $a'Ra$ both hold then $a=a'$. \end{itemize}
\begin{rem}
Following our convention, this definition is actually two definitions.
When $A^{*}$ is the free monoid on $A$ then the notion defined is
called a \emph{non-commutative broad poset. }When $A^{*}$ is the
free commutative monoid on $A$ then the notion defined is called
a \emph{commutative broad poset. }Once more, the term 'broad poset'
can be replaced throughout by either 'commutative broad poset' or
'non-commutative broad poset'. 
\end{rem}
When $(A,R)$ is a broad poset we denote $R$ by $\le$ and then the
meaning of $<$ is defined in the obvious way\emph{. }Obviously, one
has the standard constructions, for a broad relation $R$, of the
\emph{reflexive closure} $R^{r}=R\cup\{(a,a)\mid a\in A\}$ and the
\emph{transitive closure} $R^{t}=\bigcap_{R\subseteq_{t}S}S$ (where
the notation $R\subseteq_{t}S$ indicates that $S$ is a transitive
broad relation containing $R$). Lastly, if $R$ satisfies reflexivity
and transitivity then setting $a\sim b$, for $a,b\in A$, when both
$aRb$ and $bRa$ hold, defines an equivalence relation and $R_{0}=R/\sim$
inherits the broad relation structure from $R$. It follows easily
that if $R$ is any broad relation then $(R^{rt})_{0}$ is a broad
poset, called the broad poset \emph{generated }by $R$.
\begin{defn}
A function $f:A\rightarrow A'$ between broad posets is \emph{monotone}
if, for every $b\in A^{*}$ and $a\in A$, the inequality $b\le a$
implies $f(b)\le f(a)$ (where $f(b)=f(b_{1}\cdot\cdots\cdot b_{n})=f(b_{1})\cdot\cdots\cdot f(b_{n})$). 
\end{defn}
We thus obtain the categories $\mathbf{bPos_{c}}$ and $\mathbf{bPos_{\pi}}$
of, respectively, commutative and non-commutative broad posets and
monotone functions. In accordance with our convention, the term $\mathbf{bPos}$
below is meant to be replaced throughout by either $\mathbf{bPos}_{c}$
or $\mathbf{bPos_{\pi}}$. 
\begin{rem}
\label{Rem:broad preord as enriched operads}Recall that preordered
sets and monotone functions are equivalent to categories enriched
in the truth values monoidal category $V=\{F<T\}$. Similarly, commutative
broad preorders, i.e., commutative broad relations satisfying reflexivity
and transitivity, are essentially the same as symmetric operads enriched
in $V$. In the same vain, non-commutative broad preorders, i.e.,
non-commutative broad relations satisfying reflexivity and transitivity,
are essentially the same as planar operads enriched in $V$. \end{rem}
\begin{example}
\label{Example: singleton broad posets}There is a whole range of
possible broad poset structures on a singleton set, of which we mention
two. A terminal object, $*$, in $\mathbf{bPos}$ is a singleton set
$S=\{*\}$ in which, if we write $n\cdot*=*\cdot\cdots\cdot*$ for
the $n$-fold product/sum, the inequality $n\cdot*\le*$ holds for
every $n\ge0$. At the other extreme we find the broad poset $\star$
in which $n\cdot*\le*$ holds if, and only if, $n=1$. Notice that
there is, for every broad poset $A$, a bijection between the set
of monotone functions $\star\to A$ and the elements of $A$. 
\end{example}
~
\begin{example}
If $P$ is a meet semi-lattice then one can define a broad poset structure
on $P$ as follows. Given $p_{0},\cdots,p_{n}\in P$, the inequality
$p_{1},\cdots,p_{n}\le p_{0}$ holds precisely when $p_{1}\wedge\cdots\wedge p_{n}\le p_{0}$.
More generally, if $(P,\cdot,I)$ is a symmetric monoid object in
$\mathbf{Pos}$ then one similarly obtains a commutative broad poset
structure on $P$. We mention that broad posets arising from symmetric
monoid objects in $\mathbf{Pos}$ can be characterized by certain
representability conditions in a way similar to the representability
of symmetric multicategories given in \cite{MR1758246}. Similar remarks
are valid in the non-commutative case. 
\end{example}
~
\begin{example}
\label{Exam:n-corolla as broad poset}For every $n\ge0$ let $\gamma_{n}$
be a set $\{r,l_{1},\cdots,l_{n}\}$ with $n+1$ elements with the
broad poset structure in which the only inequality, other than those
imposed by reflexivity, is $l_{1}\cdot\cdots\cdot l_{n}\le r$. Note
that, for every broad poset $A$, a monotone function $\gamma_{n}\to A$
corresponds bijectively to a choice of $n+1$ elements $a_{0},\cdots,a_{n}\in A$
satisfying $a_{1}\cdot\cdots\cdot a_{n}\le a_{0}$. The broad poset
$\gamma_{n}$ is called an $n$-\emph{corolla}. \end{example}
\begin{thm}
\label{thm:comp cocomp}The category $\mathbf{bPos}$ is small complete
and small cocomplete. \end{thm}
\begin{proof}
One may easily construct all required limits and colimits directly.
A more conceptual argument uses Remark \ref{Rem:broad preord as enriched operads}
above. Since the truth values category $V=\{F<T\}$ is complete and
cocomplete it follows from general considerations of enriched operad
theory that the category $\mathbf{bPreOrd}$ of broad preorders is
small complete and small cocomplete. This suffices to construct all
small limits in $\mathbf{bPos}$. To obtain small colimits in $\mathbf{bPos}$
one needs to also employ the functor $(-)_{0}:\mathbf{bPreOrd}\to\mathbf{bPos}$,
obtained by the construction $R\mapsto R_{0}$ described above. \end{proof}
\begin{rem}
\label{Rem:cartesian structure on bPos}Below we show that the cartesian
structure on $\mathbf{bPos}$ is not closed. We thus describe explicitly
the product of two broad posets $A,B$. Such a product is obtained
by endowing the set $A\times B$ with the broad relation where $(a_{1},b_{1})\cdot\cdots\cdot(a_{n},b_{n})\le(a,b)$
holds precisely when the two inequalities $a_{1}\cdot\cdots\cdot a_{n}\le a$
and $b_{1}\cdot\cdots\cdot b_{n}\le b$ hold. The terminal object
$*$ was discussed in Example \ref{Example: singleton broad posets}.
\end{rem}
Note that the category $\mathbf{Pos}$ can be recovered, up to equivalence,
from $\mathbf{bPos}$ by slicing over the broad poset $\star$ (this
is simply the trivial observation that the unique function $A\to\star$
is monotone if, and only if, $A$ is essentially an ordinary poset).
The forgetful functor $\mathbf{bPos}/\star\to\mathbf{bPos}$ gives
an embedding $k_{!}:\mathbf{Pos}\to\mathbf{bPos}$ which is easily
seen to have a right adjoint $k^{*}:\mathbf{bPos}\to\mathbf{Pos}$.
This right adjoint $k^{*}$ sends a broad poset $(A,R)$ to the poset
$(A,S)$ where for $a,a'\in A$ holds $aSa'$ precisely when $aRa'$
holds. 

Obviously, there is a forgetful functor $\Sigma^{*}:\mathbf{bPos_{c}}\to\mathbf{bPos}_{\pi}$
(induced by the evident surjection $A^{\cdot}\to A^{+}$) whose left
adjoint $\Sigma_{!}:\mathbf{bPos}_{\pi}\to\mathbf{bPos_{c}}$ sends
a non-commutative broad poset $R$ to its abelianization.

Recall that a poset $A$ can be considered as a category $\mathscr{C}$
whose objects are the elements of $A$ and such that there is precisely
one arrow $a\rightarrow a'$ in $\mathscr{C}$ whenever $a\le a'$.
One obtains thus a functor $\mathbf{Pos}\to\mathbf{Cat}$. Similarly,
given a broad poset $B$ one can define a (symmetric or planar) operad
$\mathscr{P}$ whose objects are the elements of $B$ and such that
there is exactly one operation in $\mathscr{P}(b_{1},\cdots,b_{n};b)$
whenever $b_{1}\cdot\cdots\cdot b_{n}\le b$. In that way one obtains
the functors $\mathbf{bPos_{c}\to\mathbf{Ope}}$ and $\mathbf{bPos}_{\pi}\to\mathbf{Ope}_{\pi}$.
We summarize the properties of these constructions in the following
theorem. 
\begin{thm}
\label{thm:Main diagram}In the diagram
\[
\xymatrix{\mathbf{bPos_{c}}\ar[rrr]\ar@<2bp>[dd]^{\Sigma^{*}}\ar@<-2bp>[dr]_{k^{*}} &  &  & \mathbf{Ope}\ar@<2bp>[dd]^{\Sigma^{*}}\ar@<2bp>[dl]^{j^{*}}\\
 & \mathbf{Pos}\ar[r]\ar@<-2bp>[ul]_{k_{!}}\ar@<-2bp>[dl]_{k_{!}} & \mathbf{Cat}\ar@<2bp>[ur]^{j_{!}}\ar@<2bp>[dr]^{j_{!}}\\
\mathbf{bPos}_{\pi}\ar[rrr]\ar@<2bp>[uu]^{\Sigma_{!}}\ar@<-2bp>[ru]_{k^{*}} &  &  & \mathbf{Ope}_{\pi}\ar@<2bp>[uu]^{\Sigma_{!}}\ar@<2bp>[lu]^{j^{*}}
}
\]
all pairs of arrows are adjunctions (with left adjoint on the left
or on top) and each of the four triangles that consist of just left
or just right adjoints commutes. The horizontal arrows are embeddings
and each of the two trapezoids commutes. Moreover, each of the right
adjoints other than the left most vertical one is equivalent to the
canonical forgetful functor of a slice category. \end{thm}
\begin{proof}
We omit the proofs of the claims not given above and refer the reader
to \cite{Weiss:2010fk} for more information on some of the properties
concerning the triangles on the right. 
\end{proof}

\subsection{Tensor products}

The category $\mathbf{Pos}$ is cartesian closed with the straightforward
definition of products of posets. The internal hom, for two posets
$P,Q\in ob(\mathbf{Pos})$, is the poset $[P,Q]$ of all monotone
functions $f:P\to Q$ where $f\le g$ holds precisely when, for all
$p\in P$, the inequality $f(p)\le g(p)$ holds. This monoidal structure
is inherited from the closed cartesian structure on $\mathbf{Cat}$
along the embedding $k_{!}:\mathbf{Pos}\to\mathbf{Cat}$. It is known
that the category $\mathbf{Ope}$ is cartesian but not cartesian closed
and that it does posses a symmetric closed monoidal structure, given
by the Boardman-Vogt tensor products (\cite{MR0420609}), that restricts
along $j_{!}:\mathbf{Cat}\to\mathbf{Ope}$, to the cartesian product
of categories. We now show that similar results are true for broad
posets. 
\begin{prop}
The category $\mathbf{bPos}$ is cartesian but not cartesian closed. \end{prop}
\begin{proof}
By Theorem \ref{thm:comp cocomp} the category $\mathbf{bPos}$ has
all small products and is thus cartesian. To show that the cartesian
structure (given explicitly in Remark \ref{Rem:cartesian structure on bPos})
is not closed recall the definition of corollas from Example \ref{Exam:n-corolla as broad poset}
and consider the pushout 
\[
\xymatrix{{\star}\ar[r]\ar[d] & \gamma_{2}\ar[d]\\
\gamma_{2}\ar[r] & X
}
\]
where one of the arrows $\star\to\gamma_{2}$ choses $r\in\gamma_{2}$
and the other one chooses $l_{1}\in\gamma_{2}$. It is easy to see
that this pushout is not preserved under the functor $\gamma_{3}\times-:\mathbf{bPos}\to\mathbf{bPos}$,
thus proving the claim. \end{proof}
\begin{defn}
\label{def:tensorProductOfBroadPosets}Let $A$ and $B$ be two broad
posets. Their \emph{tensor product} $A\otimes B$ is the set $A\times B$
with the broad poset generated by the broad relation in which\end{defn}
\begin{itemize}
\item for every $a\in A$ if $b_{1}\cdot\cdots\cdot b_{n}\le b$ then $(a,b_{1})\cdot\cdots\cdot(a,b_{n})\le(a,b)$,
and
\item for every $b\in B$ if $a_{1}\cdot\cdots\cdot a_{m}\le a$ then $(a_{1},b)\cdot\cdots\cdot(a_{m},b)\le(a,b).$
\end{itemize}
Note that these defining relations guarantee that for every $a\in A$
the function $a\otimes-:B\to A\otimes B$, given by $b\mapsto(a,b)$,
is monotone and similarly that for every $b\in B$ the function $-\otimes b:A\to A\otimes B$,
given by $a\mapsto(a,b)$, is monotone. 
\begin{thm}
The category $\mathbf{bPos}$ with the tensor product of broad posets
is a symmetric closed monoidal category, and $k_{!}:\mathbf{Pos}\rightarrow\mathbf{bPos}$
is strong monoidal. \end{thm}
\begin{proof}
The broad poset $\star$ is clearly a unit for the tensor product
and it is easily verified that $\otimes$ makes $\mathbf{bPos}$ into
a symmetric monoidal category, so all that is left to do is describe
the internal hom. Given two broad posets $A$ and $B$, the set $[A,B]$
of all monotone functions $f:A\to B$ is made into a broad poset by
setting $f_{1}\cdot\cdots\cdot f_{n}\le f$ precisely when for every
$a\in A$ the inequality $f_{1}(a)\cdot\cdots\cdot f_{n}(a)\le f(a)$
holds. It is routine to verify that this broad poset is the required
internal hom. The fact that $k_{!}:\mathbf{Pos}\rightarrow\mathbf{bPos}$
is strong monoidal is trivial.
\end{proof}
Returning to the diagram of Theorem \ref{thm:Main diagram}, we see
that all of the categories there are equipped with symmetric closed
monoidal structures given by the Boardman-Vogt tensor product of operads
and tensor product of broad posets (for the corner categories), and
the cartesian structure (for the remaining two). With these monoidal
structures, each functor labeled by a $-_{!}$ is strong monoidal.
The precise monoidal behaviour of the other functors is omitted here
except for the following interesting observation. Given commutative
broad posets $A,B$ the formula
\[
\Sigma_{!}(\Sigma^{*}A\otimes\Sigma^{*}B)\cong A\otimes B
\]
holds. The same formula does not remain valid if $A$ and $B$ are
symmetric operads. This phenomenon is related to the fact that, in
the above diagram, the triangle on the left is not quite a slice of
the triangle on the right. Operads have a much greater expressive
power than broad posets do at a cost of requiring more elaborate structure.
That extra structure, in those operads that are essentially broad
posets, manifests itself by redundancy (e.g., the functor $\mathbf{bPos_{c}}\to\mathbf{Ope}$
sends $\star$ not to the terminal operad $\mathbf{Comm}$ but rather
to the operad $\mathbf{As}$ having just one object but $n!$ arrows
of arity $n$ for each $n\ge0$). The lack of this redundancy in broad
posets allows for the formula above.

\section{\label{sec:Dendroidally-ordered-sets}Dendroidally ordered sets}

This section introduces the main concept of this work, that of a dendroidally
ordered set, as a broad poset satisfying three axioms. Several of
the intuitive tree notions from Section \ref{sec:Trees-and-operads}
are established as consequences of these axioms to be used in the
subsequent sections, and a definition of the dendroidal category is
given in terms of dendroidally ordered sets and monotone functions.

\subsection{The tree formalism}

A broad poset $\le$ induces a partial order relation on the set $A^{*}$
as follows. For $a,b\in A^{*}$ we say that $a\le b$ if $b=b_{1}\cdot\cdots\cdot b_{n}$,
with each $b_{i}\in A$, and if there exist $a_{1},\cdots,a_{n}\in A^{*}$
such that $a=a_{1}\cdot\cdots\cdot a_{n}$ and such that $a_{i}\le b_{i}$
holds for each $1\le i\le n$. Notice that it is harmless to use the
same symbol $\le$ both for the broad poset on $A$ and for the induced
relation on $A^{*}$.

For $a,b\in A$ we say that $b$ is a \emph{descendent} of $a$ and
write $b\le_{d}a$, if there is some $b'\in A^{*}$ such that both
$b'\le a$ and $b\in b'$ hold. Clearly, $\le_{d}$ is a preorder
on $A$. If it is a poset then we say that the broad poset $\le$
is \emph{stratified. }A broad poset $(A,\le)$ is \emph{finite} if
the set $\le$ is finite, in which case it is automatically stratified.
For an element $a\in A$ we write $\hat{a}=\{b\in A^{*}|b<a\}$. 
\begin{defn}
Let $A$ be a broad poset and $a\in A$. If $\hat{a}=\emptyset$ then
$a$ is called a \emph{leaf}. Otherwise, if $\hat{a}$ has a maximum,
denoted by $a^{\uparrow}$, then $a$ is said to \emph{have children
}and each element in $a^{\uparrow}$ is a \emph{child }of $a$\emph{.} 
\end{defn}
Clearly it is not always the case that an element $a\in A$ is either
a leaf or has children.
\begin{rem}
Notice that it is possible that $a^{\uparrow}=x\in A$. More importantly,
it is also possible that $a^{\uparrow}=\epsilon$, the monoid unit.
In that case, $a$ is not a leaf nor does it have any element $x\in A$
as a child. Such an $a$ is called a \emph{stump}, the existence of
which is an important aspect of the formalism that agrees with the
interpretation, in operad theory, of $0$-ary operations as constants.
In each of these cases it is grammatically incorrect to say that $a$
has children but we will ignore such linguistic difficulties. \end{rem}
\begin{defn}
A \emph{dendroidally ordered set} is a finite broad poset $A$ satisfying,
for all $a_{1},\cdots,a_{n},a\in A,n\ge0$, the following three conditions.\end{defn}
\begin{itemize}
\item $\le$ is \emph{simple }in the sense that if $a_{1}\cdot\cdots\cdot a_{n}\le a$
then $a_{i}=a_{j}$ implies $i=j$. 
\item The poset $(A,\le_{d})$ has a minimal element $r_{A}$ called the
\emph{root. }
\item If $a$ is not a leaf then it has children. 
\end{itemize}
Conforming with our convention this definition actually defines two
concepts: commutative and non-commutative dendroidally ordered sets.
The use of the term 'dendroidally ordered set' is meant to be replaced
throughout by one of the two. 

When $A$ is dendroidally ordered we will also refer to its elements
as \emph{edge}s. The following useful proposition establishes several
of the intuitive concepts of trees described Section \ref{sec:Trees-and-operads}
as consequences of the axioms. 
\begin{prop}
\label{prop:elementary properties}For every dendroidally ordered
set $A$ and edges $a,a_{1},a_{2},b\in A$ the following properties
hold.\end{prop}
\begin{enumerate}
\item If $a<_{d}b$ then there is a unique child $t\in b^{\uparrow}$ for
which $a\le_{d}t$.
\item If $a_{1},a_{2}$ are descendance incomparable then the inequalities
$a_{1}\le_{d}b$ and $a_{2}\le_{d}b$ together imply the existence
of a single $c\in A^{*}$ for which both $a_{1},a_{2}\in c$ and $c\le b$
hold.
\item If $a$ is not the root then $a\in x^{\uparrow}$ holds for a unique
edge $x\in A$, called its \emph{parent}.
\item The poset $(A,\le_{d})$ has all binary joins. \end{enumerate}
\begin{proof}
~\end{proof}
\begin{enumerate}
\item $a<_{d}b$ implies that for some $x\in A^{*}$ both $a\in x$ and
$x\le b$, and so $x\in\hat{b}$. Thus, $x\le b^{\uparrow}$ which,
by definition, implies that $a\le_{d}t$ for some $t\in b^{\uparrow}$.
Uniqueness follows since the existence of two distinct such children
contradicts simplicity. 
\item Either $a_{1}=b$ or $a_{2}=b$ would imply comparability and thus
there are $t_{1},s_{1}\in b^{\uparrow}$ for which both $a_{1}\le_{d}t_{1}$
and $a_{2}\le_{d}s_{1}$ hold. If $t_{1}=s_{1}$ then repeat the argument
with $b_{1}=t_{1}$ instead of $b$ until the first $b_{n}$ where
the associated $t_{n+1}$ and $s_{n+1}$ are distinct (which must
occur since $a_{1}$ and $a_{2}$ are not comparable). Thus, we have
$a_{1},a_{2}\le_{d}t_{k}$ and $t_{k}\in t_{k-1}^{\uparrow}$ for
all $0\le k\le n$ (agreeing that $t_{0}=b$) while $a_{1}\le_{d}t_{n+1}$,
$a_{2}\le_{d}s_{n+1}$ and $t_{n+1}\ne s_{n+1}$. Using transitivity
one now easily constructs the desired tuple $c$. 
\item We may construct a sequence $t_{1}<_{d}t_{2}<_{d}t_{3}<\cdots$ such
that $t_{0}=r$ and for every $k\ge0$ holds that $a<_{d}t_{k}$ and
$t_{k+1}\in t_{k}^{\uparrow}$. Since the sequence must be finite
we obtain, for the last term $t_{m}$, that $a=t_{m}\in t_{m-1}^{\uparrow}$
. To prove uniqueness assume that $a$ is a child of both $x_{1}$
and $x_{2}$ with $x_{1}\ne x_{2}$. If $x_{1}<_{d}x_{2}$ then $x_{1}\le_{d}t$
for some $t\in x_{2}^{\uparrow}.$ One easily sees then that the case
$t=a$ contradicts with $\le$ being finite while the case $t\ne a$
contradicts with simplicity. Similarly, $x_{2}<_{d}x_{1}$ leads to
a contradiction leaving us with $x_{1}$ and $x_{2}$ incomparable.
But in that case find $c\in A^{*}$ such that $x_{1},x_{2}\in c$
and $c\le r$ to obtain a contradiction by using transitivity and
$a\in x_{1}^{\uparrow}$ and $a\in x_{2}^{\uparrow}$. 
\item We may assume that $a_{1}$ and $a_{2}$ are incomparable, and thus
none is the root, and proceed to construct their join. Let $p_{1}$
be the parent of $a_{1}$ and $p_{2}$ the parent of $a_{2}$. It
is not hard to see that $p_{1}\vee p_{2}=a_{1}\vee a_{2}$. Thus,
if $p_{1}$ and $p_{2}$ are comparable then the join is found and
otherwise the same process can be repeated. This process is bounded
by the root $r_{A}$ and thus will terminate after a finite number
of times with the desired join. 
\end{enumerate}
It is obvious that if $A\ne\emptyset$ is a finite linearly ordered
set, then the broad poset $k_{!}(A)$ is dendroidally ordered. Note
that $\le_{d}$ will have the empty join if, and only if, $A$ has
a single leaf, in which case the broad poset $A$ is essentially equal
to $k_{!}(P)$ for some linear order $P$.

\subsection{The dendroidal category}
\begin{defn}
\label{Omega - algebraic definition}The \emph{dendroidal} \emph{category}
$\Omega$ is the full subcategory of $\mathbf{bPos}$ spanned by the
dendroidally ordered sets. 
\end{defn}
Conforming with our convention we just defined two categories: $\Omega_{c}\subseteq\mathbf{bPos}_{c}$
and $\Omega_{\pi}\subseteq\mathbf{bPos}_{\pi}$, and $\Omega$ is
intended to be replaced throughout by one of the two. 

Since the simplicial category $\Delta$ is equivalent to the full
subcategory of $\mathbf{Pos}$ spanned by the finite non-empty linear
orders we may use $k_{!}$ to identify $\Delta$ as a full subcategory
of both $\Omega_{c}$ and $\Omega_{\pi}$. Now consider the diagram
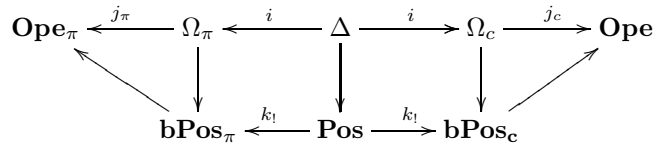
\[
\xymatrix{\mathbf{Ope_{\pi}} & \Omega_{\pi}\ar[d]\ar[l]_{j_{\pi}} & \Delta\ar[r]\ar[d]\ar[r]^{i}\ar[l]_{i} & \Omega_{c}\ar[d]\ar[r]^{j_{c}} & \mathbf{Ope}\\
 & \mathbf{bPos_{\pi}}\ar[lu] & \mathbf{Pos}\ar[r]^{k_{!}}\ar[l]_{k_{!}} & \mathbf{bPos_{c}}\ar[ur]
}
\]
where the arrows are those discussed above (and the arrows $\Omega_{c}\to\mathbf{Ope}$
and $\Omega_{\pi}\to\mathbf{Ope_{\pi}}$ are defined to make the triangles
commute). From the results below it will follow that the image of
$j_{c}$ is equivalent to the dendroidal category, and similarly the
image of $j_{\pi}$ is equivalent to the planar dendroidal category,
defined in \cite{MR2797154} in terms of operads. 
\begin{rem}
In fact the equivalence can be strengthened to an isomorphism by considering
a formalism of trees, as is done in \cite{Weiss:2010fk}, where vertices
do not exist independently of the edges but rather appear as a by
product of some structure on the edges. 
\end{rem}

\section{\label{sec:The-correspondence-between}trees and dendroidally ordered
sets}

This section studies a grafting operation for dendroidally ordered
sets that allows for precise constructions turning trees into dendroidally
ordered sets and vice versa. These constructions are the object part
of an equivalence of categories between the dendroidal category defined
in terms of operads (as in \cite{MR2797154}) and the one defined
in terms of dendroidally ordered sets.

\subsection{Grafting dendroidally ordered sets}
\begin{defn}
Let $A$ and $B$ be two dendroidally ordered sets, $\star\to A$
a leaf, and $\star\to B$ the root. A \emph{grafting }of $B$ on $A$,
denoted by $A\circ B$, is a pushout
\[
\xymatrix{\star\ar[r]\ar[d] & A\ar[d]\\
B\ar[r] & A\circ B
}
\]
in $\mathbf{bPos}$.
\end{defn}
By renaming the elements of $A$ and $B$ if needed we may assume
that $A\cap B=\{y\}$, where $y$ is the chosen leaf of $A$ and the
root of $B$. Then a grafting is obtained as the broad poset generated
by the broad relation on $A\cup B$ consisting of the relations coming
from $A$ and $B$. It is easily given explicitly: the inequality
$z\le x$ holds in $A\cup B$ if it holds in either $A$ or $B$ or
if the following holds. There exists $a_{1},a_{2}\in A^{*}$ and $b\in B^{*}$
such that $z=a_{1}ba_{2}$, $b\le y$, and $a_{1}ya_{2}\le x$. 

It is thus easily seen that the grafting of two dendroidally ordered
sets is again a dendroidally ordered set, leading to the following
corollary. 
\begin{cor}
\label{cor:pushout in omega}The grafting $B\circ A$ can be computed
in either $\mathbf{bPos}$ or $\Omega$ with isomorphic results. 
\end{cor}
By repeated grafting one can define a full grafting operation 
\[
A\circ(B_{1},\cdots,B_{n})
\]
which is simply an $n$-fold pushout. 

For a dendroidally ordered set $A$ and $a\in A$ let $A_{a}=\{a'\in A|a'\le_{d}a\}$
with the induced broad relation from $A$. It is immediate that $A_{a}$
is again dendroidally ordered. For a dendroidally ordered set $A$
with root $r$ and $r^{\uparrow}=\{a_{1},\cdots,a_{n}\}$ let $A_{root}=\{r,a_{1},\cdots,a_{n}\}$,
viewed as an $n$-corolla $\gamma_{n}$.
\begin{lem}
\label{lem:fund decom}For a dendroidally ordered set $A$ with root
$r$ and $r^{\uparrow}=\{a_{1},\cdots,a_{n}\}$ holds that $A\cong A_{root}\circ(A_{a_{1}},\cdots,A_{a_{n}})$.
Moreover, this decomposition is unique in the sense that if $A\cong\gamma_{m}\circ(A_{1},\cdots,A_{m})$
then $m=n$ and, up to reordering, $A_{a_{i}}\cong A_{i}$ for all
$1\le i\le n$.\end{lem}
\begin{proof}
We show that $A$ satisfies the universal property for the pushout
$A_{root}\circ(A_{a_{1}},\cdots,A_{a_{n}})$, of which the required
injections are evident. Suppose that $B$ is any dendroidally ordered
set with monotone function $A_{root}\to B$ and $A_{a_{i}}\to B$
making the relevant diagram commute. We need to construct an appropriate
monotone function $A\to B$. By Proposition \ref{prop:elementary properties},
for every $a\in A$ holds that if $a\notin A_{root}$ then $a\in A_{a_{i}}$
for precisely one $1\le i\le n$. Moreover, the following argument
shows that $A_{a_{i}}\cap A_{root}=\{a_{i}\}$. If $r\in A_{a_{i}}$
then it follows that $r=a_{i}$, but then $r\in r^{\uparrow}$, a
contradiction (in general $a\notin a^{\uparrow}$ holds for every
$a\in A$). If $a_{j}\in A_{a_{i}}$ and $a_{j}\ne a_{i}$ then $a_{j}\le_{d}a_{i}$
which means that there is a $b\in A^{*}$ with $a_{j}\in b$ and $b\le a_{i}$.
But then transitivity and the inequality $r\le(a_{1},\cdots,a_{n})$
will contradict the simplicity of $A$. Thus the only element of $A$
which can be in $A_{root}\cap A_{a_{i}}$ is $a_{i}$ which is clearly
there. These observations show that there is a unique function $A\to B$,
easily seen to be monotone, which is compatible with the given monotone
functions to $B$, completing the proof of the decomposition. The
uniqueness clause follows easily. \end{proof}
\begin{rem}
Combining Corollary \ref{cor:pushout in omega} and Lemma \ref{lem:fund decom}
it is seen that $\Omega$ can also be defined as the smallest full
subcategory of $\mathbf{bPos}$ containing all corollas and closed
under grafting. 
\end{rem}
Clearly this remark already implies that dendroidally ordered sets
and trees are, in a sense, the same. To furnish an exact statement
we describe constructions to turn a tree into a dendroidally ordered
set and vice versa. 

Let $T$ be a tree under any formalism that allows for a precise statement
of the fundamental decomposition exhibiting a tree $T$, essentially
uniquely, as the grafting $T=T_{root}\circ(T_{e_{1}},\cdots,T_{e_{n}})$,
as in Section \ref{sec:Trees-and-operads}. We define a dendroidally
ordered set, $[T]$, whose underlying set is $E(T)$, the set of edges
of $T$, by induction on the number $k$ of vertices in the tree $T$.
If $T=\eta$ (the tree with one edge and no leaves) then we set $[\eta]=\star$
while if $T$ is an $n$-corolla $C_{n}$ then we set $[C_{n}]=\gamma_{n}$,
covering the cases $k=0,1$. Suppose now that $T$ has more then $1$
vertex and write $T=T_{root}\circ(T_{e_{1}},\cdots,T_{e_{n}})$. We
then define $[T]=[T_{root}]\circ([T_{e_{1}}],\cdots,[T_{e_{n}}])$,
where the grafting is that of dendroidally ordered sets.

For the construction associating with any dendroidally ordered set
$A$ a tree $T$ we need the following concepts. A pair $(b,a)$ is
called a \emph{link} in a broad poset $A$ if $b<a$ and if for every
$b'\in A^{*}$ the inequalities $b\le b'<a$ imply that $b=b'$. The
number of links in a broad poset $A$ is the \emph{degree} of $A$
and is denoted by $d(A)$. When $A$ is a dendroidally ordered set
a link is called a \emph{vertex. }It can easily be shown that for
dendroidally ordered sets $A$ and $B$ the equality $d(A\circ B)=d(A)+d(B)$
holds. 

To obtain a tree $Tr(A)$ from a dendroidally ordered set $A$ we
proceed by induction on $n=d(A)$. If $n=0$ then $Tr(A)=\star$ while
if $n=1$ then $Tr(A)=\gamma_{n}$ where $n+1=|A|$, the cardinality
of the set $A$. Assume $Tr(A)$ was constructed for all $A$ with
$d(A)<n$ and let $A$ be a dendroidally ordered set with $d(A)=n$.
Then write $A=A_{root}\circ(A_{a_{1}},\cdots,A_{a_{n}})$, and let
$Tr(A)=Tr(A_{root})\circ(Tr(A_{a_{1}}),\cdots,Tr(A_{a_{n}}))$, obtained
by grafting of trees. 

Another straightforward inductive proof yields the following convenient
degree formula, where $L(A)$ denotes the set of leaves of $A$. 
\begin{lem}
\label{lem:deg formula}For every dendroidally ordered set $A$ the
equality $d(A)=|A|-|L(A)|$ holds. 
\end{lem}

\subsection{The equivalence between the operadic approach and the dendroidal
order approach}

The constructions $T\mapsto[T]$ and $A\mapsto Tr(A)$ set up a correspondence
between the trees depicted somewhat loosely in Section \ref{sec:Trees-and-operads}
and dendroidally ordered sets. We now have the categories $\Omega_{c}$
and $\Omega_{\pi}$ given above and the categories $\Omega^{O}$ and
$\Omega_{\pi}^{O}$ given in \cite{MR2797154} in terms of operads
(denoted there by $\Omega$ and $\Omega_{p}$).
\begin{thm}
There is an equivalence of categories $\Omega^{O}\cong\Omega_{c}$
and $\Omega_{\pi}^{O}\cong\Omega_{\pi}$.\end{thm}
\begin{proof}
The constructions $A\mapsto Tr(A)$ and $T\mapsto[T]$ are easily
seen to extend to functors establishing the desired equivalences.
\end{proof}
Evidently, this equivalence establishes a translation mechanism from
tree concepts to the language of dendroidally ordered sets. This is
the tree formalism we propose. From this point onwards the term 'tree'
is synonymous with 'dendroidally ordered set', and thus, conforming
with our convention, comes in two flavours: commutative and non-commutative.
Thus, 'tree' is meant to be replaced throughout by either 'commutative
tree' or 'non-commutative tree'.

\section{\label{sec:Fundamental-structure-of}face-degeneracy factorizaion}

We give a characterizes of the maximal subtrees of a given trees by
means of pruning and contraction operations and prove a fundamental
decomposition result for arrows in the dendroidal category.

\subsection{Maximal subtrees}

For the rest of this subsection fix a tree $A$ of degree $n$ and
$B\subseteq A$ a subtree (i.e., $B$ with the induced broad poset
structure is dendroidally ordered) of degree $n-1$ (such subtrees
are called \emph{maximal}). We also work under the extra assumption
that $B$ contains the root of $A$ (necessarily as the root of $B$
too). If that is not the case then the proofs below can be adapted
to yield the same bottom line, but we omit the details. 

First we notice that $L(A)\cap L(B)=L(A)\cap B$ always holds. Denote
by $k_{1}$ the number of leaves of $A$ that $B$ misses, by $k_{2}$
the number of non-leaves of $A$ that $B$ misses, by $t_{1}$ the
number of leaves in $B$ that are also leaves in $A$, and by $t_{2}$
the number of leaves in $B$ that are not leaves in $A$. By the degree
formula in Lemma \ref{lem:deg formula} we may write
\[
d(A)=|A|-|L(A)\cap B|-|L(A)-B|
\]
and 
\[
d(B)=|B|-|L(B)\cap L(A)|-|L(B)-L(A)|.
\]
Subtraction yields $1=k_{2}+t_{2}$ and we analyze all possibilities.
If $k_{2}=0$ then $B$ only misses leaves of $A$, and there is precisely
one leaf in $B$ which is not a leaf in $A$. If, as sets, $A=B$
then $B$ does not miss any edges of $A$ and the only way to then
create a new leaf is by omitting a vertex of the form $\epsilon\le x$
for a unique $x$. Otherwise, $B\subset A$ and $B$ misses at least
one leaf $l\in L(A)$. We have that $l\in e^{\uparrow}$ for a unique
edge $e$, which is not a leaf in $A$, and thus $e\in B$. We claim
that $B$ misses every child of $e$. Indeed, assume that $e_{1},\cdots,e_{k}$,
with $k>0$, are the children of $e$ not missed by $B$. In $B$
these edges are incomparable and are descendants of $e$. Thus, by
Proposition \ref{prop:elementary properties}, there is an element
$u\in B^{*}$ such that $u\le e$ and $e_{i}\in u$ for all $1\le i\le k$.
But then $u\le e$ holds in $A$ and thus $u\le e^{\uparrow}$ which
contradicts $l\in e^{\uparrow}$ being a leaf. Since $B$ only misses
leaves of $A$ we conclude that every child of $e$ is a leaf and
so $B$ misses the vertex $e\le e^{\uparrow}$ which makes $e$ into
a new leaf, the only possible new one. Thus, the case $k_{2}=0$ implies
that $B$ is obtained either by turning a stump into a leaf or by
pruning an outer cluster. If $k_{2}=1$ then $B$ misses exactly one
non-leaf $e$ and no new leaves are present in $B$. But then $k_{1}=0$
since omitting a leaf of $A$ must create a new leaf. Thus, $B$ is
obtained from $A$ by omitting a single inner edge $e$. 

We summarize these results as follows. Given a tree $A$ of degree
$n$ there are three ways to produce a maximal subtree $B$. One is
by omitting an inner edge $e\in A$, denoted by $B=A/e$. Another
is by taking $B=A$ and omitting a stump $\epsilon\le x$, and the
third one is by pruning an outer cluster $C$, denoted by $A/C$.
The second of the three will also be considered a removal of an outer
cluster. Above we established the following result. 
\begin{thm}
Let $A$ be a tree of degree $n$. If $B\subseteq A$ is a maximal
subtree then $B=A/a$ for a unique inner edge $a\in A$ or $B=A/C$
for a unique outer cluster $C$ (the meaning of 'or' should be taken
in the exclusive sense). 
\end{thm}
An inclusion $A/a\to A$ for an inner edge $a$ is called an \emph{inner
face map}. An inclusion $A/C\to A$ for an outer cluster $C$ is called
an \emph{outer face map. }
\begin{example}
\label{Exampl:A_a as outer faces}If $A$ is a dendroidally ordered
set with root $r$ and $b\in r^{\uparrow}$ then the inclusion $A_{b}\to A$
is a composition of outer face maps. To see that, notice that if $r^{\uparrow}=b$
then the vertex $r^{\uparrow}\le r$ is an outer cluster and removing
it gives $A_{b}$. Otherwise, there must be an outer cluster in $A$
which is disjoint from $A_{b}$. Removing such outer clusters one
at a time will eventually allow removing the root vertex and obtain
$A_{b}$. 
\end{example}
One more type of monotone function is the following one. Let $l=(a_{1},a_{2})$
be a unary vertex in $A$. The monotone function $\sigma_{l}:A\rightarrow A/a_{2}$
defined by 
\[
\sigma_{l}(x)=\left\{ \begin{array}{cc}
x & x\ne a_{2}\\
a_{1} & x=a_{2}
\end{array}\right.
\]
is called the \emph{degeneracy map} associated with the unary vertex
$l$. 

Considering isomorphisms of trees we note that if $f:A\rightarrow B$
is an isomorphism then $f(r_{A})=r_{B}$ and for every edge $a\in A$
the equality 
\[
f(a^{\uparrow})=f(a)^{\uparrow}
\]
holds. Obviously, the only isomorphisms in $\Omega_{\pi}$ are identities.

\subsection{Fundamental decomposition of arrows in $\Omega$}

We now prove that every arrow in $\Omega$ decomposes as a composition
of degeneracies, an isomorphism, and face maps. This result first
appeared \cite{MR2366165} without proof and more recently, with proof,
as Lemma 2.3.2 in \cite{MR2797154}. We also mention Lemma 1.3.17
in \cite{MR2764874} that establishes essentially the same result
using the polynomial functors formalism of trees. The following technical
result is easily established. 
\begin{prop}
If the map $\alpha:B\rightarrow B'$ of trees is an inner face (respectively
outer face, degeneracy, isomorphism) then for any tree $A$, the map
$A\circ\alpha:A\circ B\rightarrow A\circ B'$ is an inner face (respectively
outer face, degeneracy, isomorphism) whenever the grafting is defined.\end{prop}
\begin{lem}
Any arrow $f:A\rightarrow B$ in $\Omega$ decomposes as 
\[
\xymatrix{A\ar[r]^{f}\ar[d]^{\delta} & B\\
A'\ar[r]^{\pi} & B'\ar_{\varphi}[u]
}
\]
where $\delta:A\rightarrow A'$ is a composition of degeneracy maps,
$\pi:A'\rightarrow B'$ is an isomorphism, and $\varphi:B'\rightarrow B$
is a composition of face maps. \end{lem}
\begin{proof}
The proof is by induction on $n=d(A)+d(B)$, noting that if at any
point $d(A)=0$ then the claim is trivial. The cases $n=0$ and $n=1$
are dealt with by inspection. Assume the assertion holds for $1\le n<m$
and assume $f:A\rightarrow B$ with $|A|+|B|=m$. First assume that
$f(r_{A})=b\ne r_{B}$. In that case $f$ factors through the inclusion
$B_{b}\to B$, which by Example \ref{Exampl:A_a as outer faces} is
a composition of outer face maps. The induction hypothesis now furnishes
the desired composition. 

We now consider the case where $f(r_{A})=r_{B}$ and $f(r_{A}^{\uparrow})=r_{B}^{\uparrow}$.
Let $r_{A}^{\uparrow}=a_{1}\cdot\cdots\cdot a_{k}$ and $r_{B}^{\uparrow}=b_{1}\cdot\cdots\cdot b_{k}$
with $f(a_{i})=b_{i}$. In that case, by restricting $f$ to $A_{a_{i}}$,
one obtains the map $f_{i}:A_{a_{i}}\rightarrow B_{b_{i}}$. Let $A_{root}=\{r_{A},a_{1},\cdots,a_{k}\}$
with the broad order induced by $A$ and define $B_{root}$ similarly.
Let $f_{root}:A_{root}\rightarrow B_{root}$ be the restriction of
$f$ to $A_{root}$. The map $f$ can be written as $f_{root}\circ(f_{a_{1}},\cdots,f_{a_{k}})$.
The induction hypothesis then manufactures a decomposition of each
$f_{i}$ which can then be grafted together to produce the desired
decomposition of $f$.

The third case is when $f(r_{A})=r_{B}$ but $f(r_{A}^{\uparrow})\ne r_{B}^{\uparrow}$.
Notice that if $f(a)=r_{B}$ for some $a\in r_{A}^{\uparrow}$ then
$r_{A}^{\uparrow}=a$ (otherwise $\le$ in $B$ will not be finite)
and thus $(r_{A},a)$ is a vertex. Let $\sigma:A\rightarrow A'$ be
the degeneracy associated with it. Since $f(r_{A})=f(a)=r_{B}$ it
follows that $f$ factors through $\sigma$ as $f=f'\circ\sigma$.
The induction hypothesis applied to $f'$ together with the degeneracy
$\sigma$ produces the required decomposition of $f$. We may thus
assume further that $f(a)\ne r_{B}$ for all $a\in r_{A}^{\uparrow}$
which, together with our assumption that $f(r_{A}^{\uparrow})\ne r_{B}^{\uparrow}$,
implies that the set $I=\{x\in B\mid r_{B}<_{d}x<_{d}f(a),a\in r_{A}^{\uparrow}\}$
is non-empty and consists entirely of inner edges. Let $\hat{B}$
be the dendroidally ordered subset of $B$ obtained by removing all
of those inner elements. The inclusion $\hat{\phi}:\hat{B}\rightarrow B$
is then obviously a composition of (inner) face maps, and the map
$f$ factors as $f=\hat{\phi}\circ\hat{f}$. The induction hypothesis
applied to $\hat{f}$ together with $\hat{\phi}$ gives the desired
decomposition of $f$ and completes the proof. 
\end{proof}
\bibliographystyle{plain}
\bibliography{references}

\end{document}